\documentclass[11pt]{amsart}
\usepackage{amsmath,amsthm}
\usepackage{xypic}
\usepackage{enumerate}
\usepackage{lscape}
\usepackage{tikz-cd}
\usepackage{tikz}
\usepackage{color} 
\definecolor{darkred}{rgb}{1,0,0} 
\definecolor{darkgreen}{rgb}{0,1,0}
\definecolor{darkblue}{rgb}{0,0,1}
\usepackage{hyperref}
\hypersetup{colorlinks,
linkcolor=darkblue,
filecolor=darkblue,
urlcolor=darkred,
citecolor=darkgreen}

\usepackage{amscd}
\usepackage{amsfonts,amssymb,latexsym} 
\usepackage[all]{xy} 
\xyoption{arc}
\CompileMatrices

\newcommand{\udots}{\mathinner{\mskip1mu\raise1pt\vbox{\kern7pt\hbox{.}}
\mskip2mu\raise4pt\hbox{.}\mskip2mu\raise7pt\hbox{.}\mskip1mu}}

\newcommand{\SC}{{\mathcal{C}}}
\newcommand{\SD}{{\mathcal{D}}}
\newcommand{\SE}{{\mathcal{E}}}

\newcommand{\SH}{{\mathcal{H}}}

\newcommand{\SL}{{\mathcal{L}}}
\newcommand{\SM}{{\mathcal{M}}}
\newcommand{\SN}{{\mathcal{N}}}
\newcommand{\SO}{{\mathcal{O}}}

\newcommand{\SU}{{\mathcal{U}}}

\newcommand{\SX}{{\mathcal{X}}}

\newcommand{\PP}{\mathbb{P}}
\newcommand{\ZZ}{\mathbb{Z}}

\newcommand{\CC}{\mathbb{C}}

\newcommand{\QQ}{\mathbb{Q}}

\newcommand{\GG}{\mathbb{G}}

\newcommand{\Spec}{\operatorname{Spec}}

\newcommand{\Hom}{\operatorname{Hom}}

\newcommand{\id}{\operatorname{Id}}
\newcommand{\im}{\operatorname{Im}}
\newcommand{\too}{\longrightarrow}
\newcommand{\rk}{\operatorname{rk}}
\newcommand{\End}{\operatorname{End}}
\newcommand{\tr}{\operatorname{tr}}

\newcommand{\op}{\operatorname}

\newcommand{\simp}{\op{simp}}
\newcommand{\svb}{{\op{s-vb}}}
\newcommand{\ssvb}{\op{ss-vb}}

\newtheorem{proposition}{Proposition}[section]
\newtheorem{theorem}[proposition]{Theorem}
\newtheorem{definition}[proposition]{Definition}
\newtheorem{lemma}[proposition]{Lemma}
\newtheorem{corollary}[proposition]{Corollary}

\theoremstyle{definition}

\newtheorem{remark}[proposition]{Remark}

\numberwithin{equation}{section}

\makeatletter
\@namedef{subjclassname@2020}{\textup{2020} Mathematics Subject Classification}
\makeatother

\begin{document}

\title{Torelli theorem for moduli stacks of vector 
bundles and principal $G$-bundles}

\author[D. Alfaya]{David Alfaya}

\address{Department of Applied Mathematics and Institute for Research in Technology, ICAI School of Engineering,
Comillas Pontifical University, C/Alberto Aguilera 25, 28015 Madrid, Spain}

\email{dalfaya@comillas.edu}

\author[I. Biswas]{Indranil Biswas}

\address{Department of Mathematics, Shiv Nadar University, NH91, Tehsil Dadri,
Greater Noida, Uttar Pradesh 201314, India}

\email{indranil.biswas@snu.edu.in, indranil29@gmail.com}

\author[T. L. G\'omez]{Tom\'as L. G\'omez}

\address{Instituto de Ciencias Matem\'aticas (ICMAT), CSIC-UAM-UC3M-UCM,
Nicol\'as Cabrera 15, Campus Cantoblanco UAM, 28049 Madrid, Spain}

\email{tomas.gomez@icmat.es}

\author[S. Mukhopadhyay]{Swarnava Mukhopadhyay}

\address{School of Mathematics, Tata Institute of Fundamental Research,
Homi Bhabha Road, Mumbai 400005, India}

\email{swarnava@math.tifr.res.in}



\date{}

\begin{abstract}
Given any irreducible smooth complex projective curve $X$, of genus at least $2$, consider the moduli stack of 
vector bundles on $X$ of fixed rank and determinant. It is proved that the isomorphism class of the stack uniquely 
determines the isomorphism class of the curve $X$ and the rank of the vector bundles. The case of trivial 
determinant, rank $2$ and genus $2$ is specially interesting: the curve can be recovered from the moduli stack, 
but not from the moduli space (since this moduli space is $\mathbb{P}^3$ thus independently of the curve).

We also prove a Torelli theorem for moduli stacks of principal $G$-bundles on a curve of genus at least $3$, where 
$G$ is any non-abelian reductive group.
\end{abstract}

\maketitle

\section{Introduction}

Let $X$ be a smooth complex projective curve of genus $g$, with $g\,\ge\, 2$. The classical Torelli theorem states 
that the isomorphism class of the canonically polarized Jacobian variety $J(X)$ determines uniquely the 
isomorphism class of the curve. Natural generalizations of this problem to moduli spaces of vector bundles of 
higher rank have been studied extensively, addressing the question whether the geometry of the curve $X$ can be 
recovered from the isomorphism class of a certain moduli space of rank $r$ vector bundles on $X$.

Fix a line bundle $\xi$ on the curve $X$, and let $M^{\ssvb}(X,r,\xi)$ denote the moduli space of semistable 
vector bundles $E$ of rank $r$ on $X$ such that $\det(E)\,\cong\, \xi$.

Mumford and Newstead \cite{MN68} and Tyurin \cite{Tyu69} proved that if $X$ and $X'$ have genus at least 2 
and the degree of the determinant is odd, then $M^{\ssvb}(X,2,\xi)\,\cong\, M^{\ssvb}(X,2,\xi)$ implies that 
$X\,\cong\, X'$. This Torelli type result was then extended to moduli
spaces of vector bundles of rank $r$ when the degree of 
the determinant is coprime with $r$ by Tyurin \cite{Tyu70} and by Narasimhan and Ramanan \cite{NR75}.

Kouvidakis and Pantev \cite{KP95} obtained a Torelli theorem for genus at least three and any rank $r\,\ge\, 2$ that did not need 
the coprimality condition on rank and degree; they proved that $M^{\ssvb}(X,r,\xi)\,\cong\, M^{\ssvb}(X',r,\xi')$ implies that $X\,\cong 
\, X'$. Other proofs for this Torelli theorem for genus at least 4 have been given with different techniques by 
Hwang and Ramanan \cite{HR04}, by Sun \cite{Sun05} and by Biswas, G\'omez and Mu\~noz \cite{BGM13}.

Using the techniques in \cite{BGM13}, a 2-birational version of the Torelli was found in \cite{AB} for curves 
of genus at least 4; it was shown there that both the pair $(r,\,\pm \deg(\xi)\, \pmod{r})$ and the curve
can be recovered from the geometry of the moduli scheme. Another Torelli type theorem for moduli spaces of
rank three vector bundles with trivial 
determinant over genus 2 curves was found by Nguyen \cite{Ngu07}. On the other hand, a Torelli theorem for 
the moduli spaces of principal $G$-bundles over curves of genus at least $3$, where $G$ is a complex 
reductive group, was also proven by Biswas and Hoffmann \cite{BH12}.

In this work we study the moduli stack $\SM(X,r,\xi)$ of vector bundles over $X$ of rank $r$ with fixed 
determinant $\xi$, in particular, we prove a Torelli type theorem for this moduli stack. The objects of this 
moduli stack are pairs $(E,\,\varphi)$ where $E$ is a vector bundle of rank $r$ on $X$ and $\varphi\,:\,\det E 
\,\longrightarrow\, \xi$ is an isomorphism. An isomorphism between two objects $(E,\, \varphi)$ and 
$(E',\varphi')$ is an isomorphism $\alpha$ between $E$ and $E'$ with $\varphi\,=\, \varphi'\circ \det\alpha$. In 
particular, the automorphism group of $(E,\,\varphi)$ is the finite group of $\ZZ/r\ZZ$ when $E$ is a simple 
vector bundle.

\begin{theorem}[{Theorem \ref{thm:main}}]
\label{thm:intro}
Let $r,\,r'\,\ge\, 2$. If $X$ and $X'$ are curves of genus $g,\,g'\,\ge\, 2$ 
and $\SM(X,r,\xi)\,\cong\, \SM(X',r',\xi')$,
then $X\,\cong\, X'$ and $r\,=\,r'$.
\end{theorem}

This theorem applies to all instances where the rank and genus are both at least 2 without any additional 
coprimality conditions. Interestingly, this includes a case where the Torelli theorem
for the moduli scheme fails. Narasimhan and Ramanan \cite{NR69} proved that the moduli scheme 
$M^{\ssvb}(X,2,\SO_X)$ is isomorphic to $\PP^3$ for every genus 2 curve $X$. We show that 
that $X$ can be recovered from the complete moduli stack $\SM(X,2,\SO_X)$ and even from the substack 
$\SM^{\ssvb}(X,r,\SO_X)$ of semistable vector bundles (see Theorem \ref{thm:rank2} and Remark 
\ref{rmk:rank2}).
In the case of the moduli scheme, in Remark \ref{rmk:modulischeme} we see that the case $(g,r)\,=\,(2,2)$ is
the only exception to the Torelli theorem.

Theorem \ref{thm:intro} has been composed combining three different Torelli theorems for stacks which 
have been proven through three different strategies and using different techniques. Each of these 
theorems is valid for certain combinations of the genus of the curve and the rank of the bundle which 
do not cover the entire set of possibilities considered by Theorem \ref{thm:intro}. An additional 
argument has been made based on computations of the dimension and the Brauer class of the moduli space 
which allows us to apply selectively the appropriate version of the Torelli in each case and to combine them 
to obtain the global result summarized by Theorem \ref{thm:intro} (see Theorem \ref{thm:main}).

The first proof is based on studying the cotangent bundle to the substack of simple points and identifying it with 
a moduli stack of Higgs bundles. Working analogously to \cite{BGM13}, it is proven that the Hitchin map can be 
recovered from the geometry of this substack, and the Torelli theorem follows from a study of the geometry of the 
discriminant locus inside the Hitchin base. Due to a constraint on the codimension of a certain subvariety, the 
result works for all pairs of genus and rank $(g,\,r)$ such that $g\,\ge\, 2$ and $r\,\ge\, 2$ except for the 
three cases $(2,\,2)$, $(2,\,3)$ and $(3,\,2)$ (which correspond to the moduli schemes of the lowest dimensions 3, 
8 and 6 respectively). The details are presented in Section \ref{section:Hitchin}.

The second proof uses ``beyond GIT'' techniques based on the work of Alper, Halpern-Leistner and Heinloth 
\cite{Alp,He,HL,AlHLHe} to recover the moduli space of semistable vector bundles from the moduli stack and then 
reduces the problem to the study of the Torelli theorem for the corresponding moduli space. The limitation of this 
technique is that it can only be used in cases where the Torelli theorem is known for the corresponding moduli 
scheme. Combined with the results in \cite{AB}, we use it to obtain a Torelli theorem for curves of genus $g\ge 4$ 
in which we recover the pair $(r,\,\pm \deg(\xi)\, \pmod{r})$ in addition to the curve (see Theorem 
\ref{thm:beyondGIT}). We can also use it to prove a Torelli theorem in genus $g\,=\,3$, but it is limited to 
certain cases in $g\,=\,2$. This is studied in Section \ref{section:beyondGIT}.

This technique also allows us to prove a Torelli theorem for moduli stacks of $G$-bundles, where $G$ is any 
algebraic connected reductive complex group. Given a curve $X$, let $\SM^d_G(X)$ denote the component of the moduli stack 
of principal $G$-bundles on $X$ corresponding to a fixed $d\,\in\, \pi_1(G)$. We prove the following.

\begin{theorem}[{Corollary \ref{cor:TorellliG-bundles}}]
Let $X$ and $X'$ be smooth projective complex curves of genus at least 3, and let $G$ and $G'$ be algebraic connected 
reductive complex groups. If a moduli stack $\SM^d_G(X)$ of principal $G$-bundles over $X$ is isomorphic to a
stack $\SM^{d'}_{G'}(X')$ of principal $G'$-bundles over $X'$, then $X\,\cong\, X'$.
\end{theorem}

In Section \ref{section:rank2}, a proof for the Torelli theorem for stacks of rank 2 vector bundles with trivial 
determinant is obtained by showing that the projection of the substack of simple semistable bundles onto 
the moduli space of semistable vector bundles coincides with the quotient of the Jacobian of $X$ by 
the involution $L\,\longmapsto L^{-1}$. This is used in studying the earlier
mentioned special case of rank 2 
vector bundles with trivial determinant over a genus 2 curve.

Finally, all these results are combined in Section \ref{section:mainThm} to prove Theorem 
\ref{thm:intro}.

\section{A Torelli theorem using the Hitchin map}
\label{section:Hitchin}

Let $X$ be an irreducible smooth complex projective curve
of genus $g$, with $g\, \geq\,2$. Fix a line bundle $\xi$ on $X$. Let
$\SM\,=\, \SM(X,r,\xi)$
be the moduli stack parametrizing the vector bundles $E$ on $X$ of rank $r$ equipped with an
isomorphism $$\det(E)\,:=\,\wedge^r E\,\stackrel{\cong}{\longrightarrow}
\, \xi .$$ Let
$\SM^{\simp}(X,r,\xi)\subset \SM$ be the substack of simple points in $\SM$, i.e., the locus of vector
bundles $E$ with isomorphism $\det(E)\,\stackrel{\cong}{\longrightarrow}\, \xi$ whose automorphism group
is the group of $r$-th roots of $1\, \in\, {\mathbb C}$.

Recall that a vector bundle $E$ is said to be \emph{stable} (respectively, \emph{semistable}) if for any proper subbundle $0\ne F\subsetneq E$
$$\frac{\deg(F)}{\rk(F)} \,\,<\,\, \frac{\deg(E)}{\rk(E)} \ \ \ \, (\text{ respectively, }\,\le)$$

Let $\SM^{\svb}(X,r,\xi)\,\subset\, \SM(X,r,\xi)$ be the substack of stable vector bundles with fixed 
determinant $\xi$, and denote by $M^{\svb}(X,r,\xi)$ the corresponding moduli scheme of rank $r$ stable 
vector bundles on $X$ with fixed determinant $\xi$. Clearly, we have a quotient map 
$$\SM^{\svb}(X,r,\xi)\,\longrightarrow\, M^{\svb}(X,r,\xi).$$

The zero part of the cotangent complex of $\SM^{\simp}(X,r,\xi)$ over a vector bundle $E$ is isomorphic, 
through Serre duality, to $H^0(X,\, \End_0(E)\otimes K_X)$, where $K_X$ is the canonical line bundle 
of $X$ and $\End_0(E)\, \subset\, \End (E)$ is the subbundle of corank one defined by the
sheaf of endomorphisms of trace zero. Thus, we can interpret the total space of that sheaf as the moduli
stack $\SN^{\simp}(X,r,\xi)$ 
of pairs $(E,\,\varphi)$, where $$\varphi\ \in\ H^0(X,\, \End_0(E)\otimes K_X)$$ and $E$ is equipped with 
an isomorphism $\det(E)\,\stackrel{\cong}{\longrightarrow}\, \xi$, such that $E$ is simple. Such a pair 
$(E,\,\varphi)$ is called a \emph{Higgs bundle}, and $\varphi$ is called its \emph{Higgs field}
on $E$. Denote by 
$\SN^{\svb}(X,r,\xi)$ the substack of pairs $(E,\,\varphi)$ with $E$ being a stable vector bundle. Then 
we have a natural morphism
\begin{equation}\label{eq:svb}
\SN^{\svb}(X,r,\xi) \ \longrightarrow\ T^*M^{\svb}(X,r,\xi).
\end{equation}

A Higgs bundle $(E,\varphi)$ is said to be \emph{semistable} (respectively, \emph{stable}) if for any 
proper subbundle $0\,\ne\, F\,\subsetneq\, E$, such that $\varphi(F)\,\subseteq\, F\otimes K_X$, we have 
$$\frac{\deg(F)}{\rk(F)}\, \,\le\, \,\frac{\deg(E)}{\rk(E)} \ \ \ \, (\text{ respectively, }<).$$


Let $N(X,r,\xi)$ (respectively, $N^s(X,r,\xi)$) denote the moduli
space of semistable (respectively, stable) Higgs bundles
on $X$. We will use the following 
lemma which is a consequence of the proof of \cite[Theorem II.6.(iii)]{Fa93} or \cite[Proposition 
5.4]{BGL11}.

\begin{lemma}
\label{lemma:codimension}
Let $X$ be a curve of genus $g\,\ge\, 2$, and suppose that $r\,\ge\, 2$. Then the
codimension of $N(X,r,\xi) \backslash T^*M^{\svb}(X,r,\xi)$ in $N(X,r,\xi)$ is at least $(g-1)(r-1)$. In
particular, if $$(g,\,r)\ \not\in\ \{(2,\,2),\, (2,\,3),\, (3,\,2)\},$$ then this codimension is at least 3.
\end{lemma}

Given a stack $\SX$, let $\Gamma(\SX)$ denote the algebra of complex algebraic functions on $\SX$, i.e., 
we have $\Gamma(\SX)\,=\,\Hom_{(\op{Stacks})}(\SX,\,\CC)\,=\, H^0(\SX,\SO_{\SX})$.

\begin{lemma}\label{le1}
The equality
$$\Gamma(\SN^{\svb}(X,r,\xi))\,=\,\Gamma(T^*M^{\svb}(X,r,\xi))$$
holds.
\end{lemma}

\begin{proof}
It follows immediately from the fact that the morphism \eqref{eq:svb} is a good
moduli and hence the global functions are the same.

We prove that the morphism \eqref{eq:svb} is a good moduli as follows.
The morphism $\SM^{\svb}(X,r,d)\too M^{\svb}(X,r,d)$ from the moduli stack of
stable vector bundles to its moduli space is a $B\GG_m$-gerbe, meaning
that it is locally a product $U\times B\GG_m$ where $U$ is an \'etale covering
of $M^s(X,r,d)$. Using this and the following Cartesian diagram 
(where $\mathcal{P}ic(X)$ is
the algebraic stack parametrizing line bundles)
$$
\xymatrix{
{\SM^{\svb}(X,r,\xi)} \ar[r] \ar[d] & \SM^{\svb}(X,r,d) \ar[d]\\
{\Spec \CC} \ar[r]^{\xi} & {\mathcal{P}ic(X)}
}
$$
it is easy to see that $\SM^{\svb}(X,r,\xi)\too M^{\svb}(X,r,\xi)$ is a 
$\ZZ/r\ZZ$-gerbe.
The moduli space $N^{\svb}(X,r,\xi)$ of Higgs bundles whose underlying vector bundle is
stable is actually a vector bundle over $M^{\svb}(X,r,\xi)$. 
Therefore, this Cartesian diagram
$$
\xymatrix{
{\SN^\svb(X,r,\xi)} \ar[r] \ar[d] & N^\svb(X,r,\xi) \ar[d]\\
{\SM^{\svb}(X,r,\xi)} \ar[r] & {M^{\svb}(X,r,\xi)}
}
$$
shows that 
$$
\SN^\svb(X,r,\xi) \ \too\ N^\svb(X,r,\xi)\ =\ T^*M^{\svb}(X,r,\xi)
$$ 
is a $\ZZ/r\ZZ$-gerbe, and therefore it is a good moduli.
\end{proof}

Denote the Hitchin base as $W\,:=\,\bigoplus_{k=2}^r H^0(X,\, K^{\otimes k}_X)$, and write
$W_k\,=\,H^0(X,\, K^{\otimes k}_X)$ for $k\,=\,2,\,\cdots,\,r$.
Let 
$$
H\,:\,N(X,r,\xi)\,\longrightarrow \,W
$$
be the Hitchin map. We also define the Hitchin map for the moduli
stack
$$\SH\,:\,\SN^{\simp}(X,r,\xi) \,\longrightarrow\, W$$
sending a family $(\SE,\,\Phi)$ over $X\times T$ to the map
$$\sum_{k=2}^r (-1)^k \tr(\wedge^k\Phi) \,:\,T\,\longrightarrow \,W.$$
For each $s\,=\,(s_2,\,\ldots,\,s_r)\,\in\, W\,=\,\bigoplus_{k=2}^r H^0(X,\, K_X^k)$, the equation
$$t^r+\sum_{k=2}^r s_k t^k\,=\,0$$
defines a spectral curve $X_s$ in the total space of the line bundle $T_X^*$.

\begin{lemma}\label{le2}
If $(g,\,r)\,\ne\, (2,\,2)$, then the equality
$$\Gamma(\SN^{\simp}(X,r,\xi))\ =\ \Gamma(T^*M^{\svb}(X,r,\xi))$$
holds.
\end{lemma}

\begin{proof}
We will show that the restriction of global functions from $\SN^{\simp}(X,r,\xi)$ to 
$\SN^{\svb}(X,r,\xi)$ is an isomorphism when $(g,\,r)\,\ne\, (2,\,2)$, 
and then the result follows from Lemma \ref{le1}.

The first step is to show that any global function on $\SN^{\svb}(X,r,\xi)$ 
can be extended to $\SN^{\simp}(X,r,\xi)$. By Lemma
\ref{lemma:codimension}, the codimension of the complement of
$T^*M^{\svb}(X,r,\xi)$ in $N(X,r,\xi)$ is at least $2$, so Hartogs'
Theorem implies that
$\Gamma(T^*M^{\svb}(X,r,\xi))\,=\,\Gamma(N(X,r,\xi))$. The algebra 
of functions $\Gamma(N(X,r,\xi))$ is
generated by components of the Hitchin map \cite{Hi}, so the
algebra of functions on $T^*M^{\svb}(X,r,\xi)$ is also generated by
the components of the Hitchin map. Using Lemma \ref{le1} we conclude
that the algebra of global functions on $\SN^{\svb}(X,r,\xi)$ is 
generated by the components of the Hitchin map.
These functions are clearly well
defined over arbitrary families of Higgs fields over vector bundles
on $X$, so they extend to algebraic functions on
$\SN^{\simp}(X,r,\xi)$.

Finally, the extensions are unique because 
$\SN^{\simp}(X,r,\xi)$ is integral. Indeed, it is a vector bundle over 
$\SM^{\svb}(X,r,\xi)$, which is integral because it is actually an open substack
of the integral stack $\SM(X,r,\xi)$.
\end{proof}

\begin{corollary}\label{cor:recoverHitchin}
There exists an algebraic isomorphism $$\Spec(\Gamma(\SN^{\simp}(X,r,\xi)))\,\stackrel{\cong}{\longrightarrow}\, W$$
such that the composition of maps
$$\SN^{\simp}(X,r,\xi) \,\longrightarrow \,\Spec(\Gamma(\SN^{\simp}(X,r,\xi)))\,\stackrel{\cong}{\longrightarrow}\, W$$
coincides with the Hitchin map $\SH\,:\,\SN^{\simp}(X,r,\xi) \,\longrightarrow\, W.$
\end{corollary}

Let $\SD\,\subset\, W$ denote the discriminant locus, i.e., the locus of all $s\,=\,(s_i)\,\in\, W$ such that the
corresponding spectral curve $X_s\,\subset \,\op{Tot}(T_X^*)$ is singular.

\begin{lemma}\label{lemma:pushforwardSmooth}
Let $(E,\,\varphi)$ be a Higgs bundle whose spectral curve is integral. Then $(E,\,\varphi)$ does not have any nontrivial invariant subbundle and, in particular, it is a stable Higgs bundle.
\end{lemma}

\begin{proof}
Let $X_s$ be the spectral curve associated to $(E,\,\varphi)$. Then $(E,\,\varphi)$ is the pushforward of
a rank 1 torsion-free sheaf $L$ on $X_s$. Assume that $F$ is a nonzero subbundle preserved by $\varphi$. Since $F$ is
invariant, the characteristic polynomial of the restriction $\varphi|_F\,:\,F\,
\longrightarrow\, F\otimes K_X$ divides the characteristic polynomial of $\varphi$.
Consequently, the spectral curve associated to $(F,\,\varphi|_F)$ is a closed
subscheme of $X_s$. Since $X_s$ is integral, the spectral curve for $(F,\,\varphi|_F)$ must be the
entire $X_s$. But this implies that $\rk(F)\,=\,\rk(E)$ and, thus, we have $F\,=\,E$.
\end{proof}

\begin{lemma}\label{lemma:recoverDiscriminant1}
Let $\gamma\,: \,\PP^1\,\longrightarrow \,\SN^{\simp}(X,r,\xi)$ be a map whose image contains at least two non-isomorphic points. Then
the image of $\SH\circ \gamma$ is a point in the discriminant locus $\SD$.
\end{lemma}

\begin{proof}
First of all, as $\SH\circ\gamma \,: \,\PP^1\,\longrightarrow\, W$ is a map from $\PP^1$ to an affine space, its image must be a point
$s\,\in\, W$. Suppose that $$s\, \notin\, \SD .$$ Then the curve $X_s$
is smooth. By Lemma \ref{lemma:pushforwardSmooth}, the map $\gamma$
factors though 
the substack
$\SN'\,\hookrightarrow \,\SN^{\simp}(X,r,\xi)$ of Higgs bundles
$(E,\varphi)$ such that $(E,\varphi)$ is stable and $E$ is simple. 
Furthermore, the composition $\SH\circ \gamma$ factors through the moduli 
scheme of stable Higgs bundles: 
$$
\PP^1 \,\stackrel{\gamma}{\longrightarrow}\, \SN' \,\longrightarrow
\,N^s(X,r,\xi) \,\longrightarrow\, W.
$$ 
The preimage of $s\,\in\, W$ 
in $N^s(X,r,\xi)$ is isomorphic to the Prym variety of line bundles over $X_s$ whose pushforward has determinant $\xi$, so it is an abelian 
variety. Since there is no nonconstant map from $\PP^1$ to an abelian variety,
the image of $\gamma$ in the moduli space $N^s(X,r,\xi)$ is a single point. Thus, all the points in the image of 
$\gamma$ in the stack $\SN'$ must be isomorphic. This contradicts the hypothesis that its image contains at least two non-isomorphic 
points. This completes the proof of the lemma.
\end{proof}

\begin{lemma}\label{lemma:P1Discriminant}
Assume that $g,\, r\,\ge\, 2$, and $(g,\,r)\,\not\in\, \{(2,\,2),\,(2,\,3),\,(3,\,2)\}$. For a general
point $s\,\in\, \SD$ there exists a non-constant morphism (given 
explicitly in the proof bellow) 
$$
\gamma'\,:\,\PP^1 \,\longrightarrow \,
T^*M^{\svb}(X,r,\xi)
$$ 
such that $\im(H\circ \gamma')\,=\,s$.
\end{lemma}

\begin{proof}
We can follow the same proof as in \cite[Proposition 3.1]{BGM13}, incorporating the codimension bound given
by Lemma \ref{lemma:codimension}. By \cite[Remark 1.7]{KP95}, there exists a Zariski open subset $\SD^0\,
\subset \, \SD$ such that each point $s\,\in\, \SD^0$ corresponds to a spectral curve $X_s$ which is an
irreducible nodal curve with a single node. Moreover, as a consequence of \cite[Proposition 3.2]{BGM13},
the Hitchin discriminant $\SD$ is irreducible, so the open subset $\SD^0$ is actually dense.
Let $$H\ :\ N(X,r,\xi)\ \longrightarrow \ W$$
denote the Hitchin map for the moduli space of semistable Higgs bundles.

Let $\pi\,:\,X_s\,\longrightarrow\, X$ be the projection from the spectral curve.
The fiber $H^{-1}(s)$ parametrizes torsion free sheaves $L$ on $X_s$ such that
$\pi_*L$ is a vector bundle on $X$ of determinant $\xi$. 
Torsion Free sheaves of rank 1 on nodal curves have been studied by Usha Bhosle, 
using the notion of generalized parabolic bundles. We will now recall the results
 that we will need (for details, see the proof of \cite[Proposition 2.2]{Bh92}). 
The results of Bhosle show that the fiber $H^{-1}(s)$ is a fibration over a closed subscheme of
the Jacobian $J(\widetilde{X}_s)$ of the normalization 
$p\,:\,\widetilde{X}_s\,\longrightarrow\, X_s$ 
with fiber isomorphic to a rational curve with one node.
To see this, we first consider a line bundle on $X_s$. It can be described
by a line bundle $\widetilde{L}\,\in\, J(\widetilde{X}_s)$ and an isomorphism between the 
fibers over the two points $x_1$,\, $x_2\,\in\, \widetilde{X}_s$ 
mapping to the node of $X_s$. This isomorphism
can be given by its graph 
$\Gamma\,\subset\, \widetilde{L}_{x_1}\oplus \widetilde{L}_{x_2}$. The 
corresponding line bundle $L$ on $X_s$ fits in a short exact sequence
$$
0 \,\too\, L_{\Gamma} \,\too\, p_*{\widetilde{L}} \,\too\, 
(\widetilde{L}_{x_1}\oplus \widetilde{L}_{x_2})/\Gamma \,\too\, 0.
$$
Note that $\Gamma$ is a line in $\widetilde{L}_{x_1}\oplus \widetilde{L}_{x_2}$ which projects
isomorphically to both $\widetilde{L}_{x_1}$ and $\widetilde{L}_{x_2}$. If we allow $\Gamma$ to
become $\widetilde{L}_{x_1}$ or $\widetilde{L}_{x_2}$, then $L_\Gamma$ is no longer a line bundle, but
it is torsion free. 

We thus obtain, for each $s\in \SD^0$ and line bundle $\widetilde{L}$ on $\widetilde{X}_s$, 
a family of torsion free sheaves on $X_s$ parametrized by 
$\PP^1=\PP(\widetilde{L}_{x_1}\oplus \widetilde{L}_{x_2})$
$$
0 \,\too\, \SL \,\too\, p^*_{X_s} p_* \widetilde{L} 
\,\too\, \SO_{x_0 \times \PP^1}(1)\,\too\, 0,
$$
where $p_{X_s}$ is the projection of $X_s\times \PP^1$ to the first factor.
As we vary over all possible line bundles on $\widetilde{X_s}$ and
points in $\PP(\widetilde{L}_{x_1}\oplus \widetilde{L}_{x_2})$ we obtain all possible
torsion free sheaves on $X_s$. The condition that 
the vector bundle $\pi_*L$ on $X$ has determinant $\xi$ picks a 
closed subset of $J(\widetilde{X_s})$. Different points in 
$\PP(\widetilde{L}_{x_1}\oplus \widetilde{L}_{x_2})$ will give different isomorphic classes
of torsion free sheaves except that the points corresponding to the
two lines $\widetilde{L}_{x_1}$ and $\widetilde{L}_{x_2}$ give isomorphic torsion free sheaves.
This is the reason why $H^{-1}(s)$ is a fibration with fibers equal
to nodal rational curves.
Taking the pushforward of the previous sequence, the family $\SL$ of torsion free sheaves
becomes a family of Higgs bundles $(\SE,\Phi)$ on $X$ with $\SE$ given by
\begin{equation}\label{eq:ratline}
0\,\too\, \SE=(\pi\times \id_{\PP^1})_*\SL \,\too\,(\pi\times \id_{\PP^1})_*p^*_{X_s} p_* \widetilde{L} 
\,\too\, \SO_{\pi(x_0) \times \PP^1}(1)\,\too\, 0
\end{equation}
Furthermore, it follows from this sequence that $\det(\SE) \;\cong\; p^*_X\xi$.

Lemma
\ref{lemma:codimension} implies that the codimension of the complement of $T^*M^{\svb}(X,r,\xi)$ in
$N(X,r,\xi)$ is at least $3$, so, intersecting it with the divisor $H^{-1}(\SD)$, we obtain that the
codimension of the
complement of $H^{-1}(\SD)\cap T^*M^{\svb}(X,r,\xi)$ inside $H^{-1}(\SD)$ must be at least $2$. Since
the curves in $\SD^0$ are all integral, by Lemma \ref{lemma:pushforwardSmooth} we have
$H^{-1}(s)\,\subset\, N^s(X,r,\xi)$ for all $s$ in the dense subset $\SD^0\,\subset\, \SD$. By
\cite[Theorem II.5]{Fa93}, the restriction of the Hitchin map to $N^s(X,r,\xi)$ is equidimensional.
Thus, for a general $s\,\in \,\SD^0$, the codimension of 
the complement of $H^{-1}(s)\cap T^*M^{\svb}(X,r,\xi)$ inside $H^{-1}(s)$ is at least 2. 

As it was mentioned above, $H^{-1}(s)$ is a fibration by nodal rational curves
(dimension 1), therefore, there exists complete rational curves
$\gamma'\,:\,\PP^1 \,\longrightarrow\,H^{-1}(s)\cap T^*M^{\svb}(X,r,\xi)$.

\end{proof}

\begin{lemma}
\label{lemma:recoverDiscriminant2}
The morphism given in Lemma \ref{lemma:P1Discriminant}
can be lifted to the moduli stack, i.e., to a morphism
$$
\gamma\,:\, \PP^1 \,\longrightarrow \,\SN^{\simp}(X,r,\xi)
$$
such that the image of the composition $\SH\circ\gamma$ is the 
point $s$,
and the image of $\gamma$ contains at least two non-isomorphic points.
\end{lemma}

\begin{proof}
The morphism in Lemma \ref{lemma:P1Discriminant} is given by the
explicit family $(\SE,\Phi)$ given in \eqref{eq:ratline}. There is an
isomorphism $\det(\SE) \;\cong\; p^*_X\xi$, and hence a 
morphism $\gamma$ to the moduli stack. 

By construction, the map $\gamma'$ is nonconstant, so the above morphism $\gamma$ has at least two 
non-isomorphic points in its image.
\end{proof}

\begin{corollary}
\label{cor:recoverDiscriminant}
Let $X$ be an irreducible smooth complex projective curve of genus $g\,\ge \,2$. Suppose that $r\,\ge\, 2$
and $(g,\,r)\,\not\in\, \{(2,\,2),\,(2,\,3),\,(3,\,2)\}$. Let $\Gamma$ be the
space of all maps $$\gamma\,:\,\PP^1 \,\longrightarrow\, \SN^{\simp}(X,r,\xi)$$ whose image contains
at least two non-isomorphic points. Then
the Hitchin discriminant $\SD$ is the algebraic closure of the subset
$$\SD_\Gamma\ =\ \{\im(\SH \circ \gamma) \, \,\big\vert \,\, \gamma \,\in \,\Gamma \} \ \subset\ W.$$
\end{corollary}

\begin{proof}
By Lemma \ref{lemma:recoverDiscriminant1} we have $\SD_\Gamma\,\subseteq\, \SD$. Moreover, Lemma
\ref{lemma:recoverDiscriminant2} implies that $\SD_\Gamma$ contains a dense open subset of $\SD$. Therefore,
the closure of $\SD_\Gamma$ in $W$ is the entire discriminant $\SD$.
\end{proof}

\begin{theorem}\label{thm:HitchinMap}
Let $X$ and $X'$ be two irreducible smooth complex projective curves of genus $g$ and $g'$
respectively, with $g,\,g'\,\ge\, 2$. Let $r,\,r'\,\ge\, 2$ such that $(g,\,r),\, (g',r')\, \not\in\,
\{(2,\,2),\,(2,\,3),\,(3,\,2)\}$. Fix line bundles $\xi$ and $\xi'$ on $X$ and $X'$ respectively. Let
$$\Psi\,:\,\SM(X,r,\xi) \,\longrightarrow\, \SM(X',r',\xi')$$ be an
isomorphism between the corresponding moduli stacks of vector bundles with fixed determinant. Then
$r\,=\,r'$ and $X\,\cong\, X'$.
\end{theorem}

\begin{proof} Let $\Psi\,:\,\SM(X,r,\xi) \,\longrightarrow\, \SM(X',r',\xi')$ be an isomorphism of 
stacks. Then it preserves the locus of objects with zero-dimensional stabilizers. Any vector bundle of rank $r$
with fixed determinant admits a natural action of the group of $r$-th roots of unity by dilation. Thus, the size of the stabilizer of any object in the moduli stack $\SM(X,r,\xi)$ with a zero-dimensional stabilizer is at least $r$, and objects whose stabilizer has the minimum possible size $r$ are simple. Since these exist (for instance, stable objects are simple), we can characterize the locus $\SM^{\simp}(X,r,\xi)$ inside $\SM(X,r,\xi)$ as
the locus of objects with minimal stabilizer. As this property is preserved through the isomorphism $\Psi$, the
map $\Psi$ restricts to an isomorphism
$$\Psi^{\simp}\,:\,\SM^{\simp}(X,r,\xi)\,\longrightarrow\, \SM^{\simp}(X',r',\xi')$$
between the corresponding loci of simple objects. This map $\Psi^{\simp}$ induces an isomorphism between the corresponding cotangent complexes. As the moduli stack of
bundles is smooth, both complexes are concentrated in orders $0$ and $1$ and there is an isomorphism
$$d^1((\Psi^{\simp})^{-1})\,:\,\SN^{\simp}(X,r,\xi) \,\longrightarrow \,\SN^{\simp}(X',r',\xi').$$
Let
$$W\,=\,\bigoplus_{k=2}^r H^0(X,\ K^{\otimes k}_X),\, \ \ \ \ W'\,=\,
\bigoplus_{k=2}^{r'}H^0(K^{\otimes k}_{X'}).$$
By Corollary \ref{cor:recoverHitchin}, there exists an isomorphism $f\,:\,W\,\stackrel{\cong}{\longrightarrow}\, W'$ such that the
following diagram is commutative:
\begin{equation}\label{ef}
\xymatrix{
\SN^{\simp}(X,r,\xi) \ar[rrr]^{d^1((\Psi^{\simp})^{-1})} \ar[d]_{\SH} &&& \SN^{\simp}(X',r',\xi') \ar[d]^{\SH'}\\
W \ar[rrr]^{f}& && W'
}
\end{equation}
As the map $d^1((\Psi^{\simp})^{-1})$ is $\CC$-linear, the map $f$ in \eqref{ef} is $\CC^*$-equivariant for the
$\CC^*$ actions on $W$ and $W'$ making the Hitchin maps $\CC^*$-equivariant; more precisely,
the $\CC^*$ action is the diagonal weighted action
$$\lambda \cdot (s_2,\,\cdots,\,s_r) \,=\, (\lambda^2 s_2,\,\cdots,\, \lambda^r s_r).$$
In particular, $f$ preserves the filtrations of subspaces of $W$ and $W'$ in terms of the asymptotic decay of
the corresponding $\CC^*$-actions:
$$W\,=\,W_{\ge 2}\,\supsetneq \,W_{\ge 3} \,\supsetneq\, \cdots \,\supsetneq\, W_{\ge r}\,\supsetneq\, 0,$$
$$W'\,=\,W'_{\ge 2}\,\supsetneq \,W'_{\ge 3} \,\supsetneq\, \cdots \,\supsetneq\, W'_{\ge r'}\,\supsetneq\, 0,$$
where $W_{\ge k} \,=\, \bigoplus_{j=k}^r H^0(X,\, K_X^k) \,= \,\bigoplus_{j=k}^r W_k$ and
$W'_{\ge k} \,=\, \bigoplus_{j=k}^{r'} H^0(X',\, K_{X'}^k) \,= \,\bigoplus_{j=k}^{r'} W'_k$.

Observe that the length of the filtrations of $W$ and $W'$ are $r-1$ and $r'-1$ respectively, and so 
we conclude that $r\,=\,r'$. Moreover, $f$ sends $W_r\,\subset\, W$ to $W'_{r'}\,=\,W_r'\,\subset\, W'$ and, as the 
$\CC^*$-action is homogeneous of degree $r$ in $W_r$ and $W_r'$, and $f$ is $\CC^*$-equivariant, we conclude that
$f$ restricts to a linear map $f_r\,:\,W_r\,\longrightarrow\, W_r'$.

On the other hand, as $d^1((\Psi^{\simp})^{-1})$ is an isomorphism, it induces a bijection between the 
set $\Gamma$ of maps $\PP^1\,\longrightarrow\, \SN^{\simp}(X,r,\xi)$ whose image contains non-isomorphic 
points and the set $\Gamma'$ of maps $\PP^1\,\longrightarrow \,\SN^{\simp}(X',r',\xi')$ whose image 
contains non-isomorphic points. By Corollary \ref{cor:recoverDiscriminant}, this implies that the map 
$f\,:\,W \,\longrightarrow\, W'$ sends the Hitchin discriminant $\SD\,\subset\, W$ to the Hitchin 
discriminant $\SD'\,\subset\, W'$.

As $f(\SD)\,=\,\SD'$ and $f(W_r)\,=\,W'_r$, we have $f(\SD\cap W_r)\,=\,\SD'\cap W_r'$. Let
$$\SC\,=\,\PP(\SD\cap 
W_r)\,\subset\, \PP(W_r) \,\, \text{ and }\ \ \SC'\,=\,\PP(\SD'\cap W_r')\,\subset\, \PP(W_r').$$
Since $f_r\,:\,W_r\,
\longrightarrow\, W_r'$ is linear and $f_r(\SD\cap W_r)\,=\,\SD'\cap W_r'$, we conclude that $f_r$ induces
an isomorphism between $\PP(W_r)$ and $\PP(W_r')$ sending 
$\SC$ to $\SC'$. Then, it induces an isomorphism between the corresponding dual varieties $\SC^\vee\,\subset\, 
\PP(W_r^\vee)$ and $(\SC')^\vee \,\subset\, \PP((W_r')^\vee)$. By \cite[Proposition 4.2]{BGM13}, we have
$\SC^\vee \,\cong\, X\,\subset\, \PP(W_r^\vee)$ and $(\SC')^\vee\,\cong\, X'\,\subset\, \PP((W_r')^\vee)$.
This completes the proof.
\end{proof}

\section{``Beyond GIT'' techniques}\label{section:beyondGIT}

As before, let $\SM^{\ssvb}(X,r,\xi)\,\subset\, \SM(X,r,\xi)$ be the substack of semistable vector bundles 
with fixed determinant $\xi$, and denote by $M^{\ssvb}(X,r,\xi)$ the corresponding projective moduli 
scheme of rank $r$ semistable vector bundles on $X$ with fixed determinant $\xi$. As mentioned in the 
introduction, there exist multiple Torelli type theorems for the moduli scheme of vector bundles of rank 
$r\ge 2$ \cite{MN68,Tyu69,Tyu70,NR75,HR04,Sun05,BGM13,AB,BH12,Ngu07} showing that if $M^{\ssvb}(X,r,\xi)$ 
is isomorphic to $M^{\ssvb}(X',r',\xi')$ for some irreducible, smooth projective curve $X'$ and a line 
bundle $\xi'$ on $X'$, then $X$ is isomorphic to $X'$.

Thus if we want to show that the moduli stack $\SM(X,r,\xi)$ uniquely determines $X$, it is enough to 
recover the projective moduli scheme $M^{\ssvb}(X,r,\xi)$ from the stack $\SM(X,r,\xi)$ (provided
$M^{\ssvb}(X,r,\xi)$ uniquely determines $X$). 

One way to recover the moduli substack of semistable bundles is to use ideas from ``beyond GIT'', the theory developed by Alper, Halpern-Leistner and Heinloth 
\cite{Alp,He,HL,AlHLHe}. See \cite{ABBLT} for an exposition in the case of vector bundles.

In this theory, the notion of $\SL$-stability on a stack $\SM$ is defined, where $\SL$ is a line bundle on $\SM$. For this, we first need to introduce the 
quotient stack $\Theta\,=\,[\Spec(\mathbb{C}[t])/\GG_m]$, with the standard action
of $\GG_m$ on the line $\Spec \mathbb{C}[t]$. There are two orbits: $t\,=\,0$ and $t
\,\neq\, 0$ and therefore the stack $\Theta$ has two points which we call
$t\,=\,0$ (with automorphism group $\GG_m$) and $t\,=\,1$ (with trivial automorphism group). 

A \emph{filtration} of a point $x\,\in\, \SM$ is a morphism $f\,:\,\Theta\,\longrightarrow\,
\SM$ together with an isomorphism $f(1)\,\cong\, x$. We note that the name ``filtration'' comes from the fact
that, if $\SM$ is the moduli stack of coherent sheaves then, by the Rees construction, giving such a morphism is equivalent to giving a $\ZZ$-indexed 
filtration of the sheaf $f(1)$, and the point $f(0)$ corresponds to the 
associated graded sheaf.

The line bundle $f^*\SL$ on $\Theta$ can be thought of as a $\GG_m$-equivariant line bundle on $\Spec 
\mathbb{C}[t]$. Let $\operatorname{wt}(f^*\SL|_0)$ be the weight of this equivariant line bundle on the 
fiber over $t\,=\,0$.

\begin{definition}[{$\mathcal{L}$-semistability \cite[Definition 1.2 and Remark 1.3]{He},\cite{HL}}]
\label{def:l-ss}
A point $x\,\in\, \SM$ in an algebraic stack $\SM$ is called $\SL$-semistable if for all filtrations
$f\,:\,\Theta\,\longrightarrow\, \SM$ of $x$, we have
$$
\operatorname{wt}(f^* \SL|_0)\ \leq\ 0.
$$
\end{definition}

\begin{remark}
\label{different-l}
Note that the weight $\operatorname{wt}(f^* \SL|_0)$ 
is given by the group homomorphism
$$
f^* (\cdot)|_0\,\,:\,\,\op{Pic}(\SM) \,\,\too\,\, \op{Pic}(B\GG_m)\,\,\cong\,\, \ZZ .
$$
This implies the following:
\begin{itemize}
\item The notion of $\SL$-semistability depends only on the 
class of $\SL$ modulo torsion. 

\item The notion of $\SL$-stability only depends on the class of 
$\SL$ in the quotient $\op{Pic}(\SM)/\op{Pic}^0(\SM)$,
where $\op{Pic}^0(\SM)$ is the connected component of the identity element.

\item Note that 
$\operatorname{wt}(f^* \SL^a|_0)\,=\,a\operatorname{wt}(f^* \SL|_0)$ and
then, if $a\,>\,0$, a point is $\SL$-semistable if and only if it is 
$\SL^a$ semistable. 

\item Therefore, we can define $\SL$-semistability for any rational line
bundle $\SL\,\in\, \op{Pic}{\SM}\otimes \mathbb{Q}$, and it depends
only on the line $\QQ_{>0}\SL$.

\item If we precompose $f\,:\,\Theta\,\too\, \SM$ with the map
$\Theta\, \stackrel{[n]}{\too}\, \Theta$ defined by $t\,\longmapsto\, t^n$, 
then the weight 
$\operatorname{wt}(f^* \SL|_0)$ gets multiplied by $n$, so its sign does
not change.

\end{itemize}
\end{remark}

Let $\SL_{\det}$ be the determinant line bundle on the moduli stack of vector bundles 
$\SM(X,r,\xi)$ whose fiber over a vector bundle $E$ is $\det(H^0(E))^{-1}\otimes \det (H^1(E))$. More 
precisely, for any $f\,:\,T\,\longrightarrow\, \SM(X,r,\xi)$ corresponding to a vector bundle $\SE$ on $X\times T$, we have
$f^*\SL_{\det}\,=\,\det(Rp^{}_T{}^{}_* \SE)^{-1}$.

Recall that $\op{Pic}(\SM(X,r,\xi))\otimes \QQ\,\, \cong\,\,\QQ$ with $\SL_{\det}$
being a generator. This was proved for 
the moduli functor and the moduli scheme in \cite{DN}. For a detailed proof in the case of the moduli stack, 
valid for any genus, see \cite[Proposition 4.2.3 and Theorem 4.2.1]{BH10} (see also \cite[Lemma 7.8 and Remark 7.11]{BL}, \cite{Fa94} and \cite{Fa03}).

\begin{proposition}\label{L-semistability}\mbox{}
\begin{itemize}
\item If $a\,<\,0$ is an integer, then all points $x\,\in\, \SM(X,r,\xi)$ are $\SL_{\det}^a$-unstable.

\item If $a\,=\,0$, then all points $x\,\in\, \SM(X,r,\xi)$ are $\SL_{\det}^a$-semistable.

\item If $a\,>\,0$ is an integer, then $x\,\in\, \SM(X,r,\xi)$ is $\SL_{\det}^a$-semistable if and only if
the vector bundle $E$ corresponding to $x$ is semistable in the usual sense.
\end{itemize}
\end{proposition}

\begin{proof}
Giving a morphism $\Theta\,=\,[\Spec(\mathbb{C}[t])/\GG_m]\,\longrightarrow\, \SM(X,r,\xi)$ is equivalent to giving a $\GG_m$-equivariant
morphism $\Spec(\mathbb{C}[t])\,\longrightarrow\, \SM(X,r,\xi)$, and this is equivalent to giving a vector bundle $\SE$ 
on $\Spec(\mathbb{C}[t])\times X$ together with a lift of the $\GG_m$ action on $\Spec(\mathbb{C}[t])$. By
the Rees construction, this is equivalent to giving a $\ZZ$-indexed filtration $E_\bullet$ of
a vector bundle $E$ on $X$, with 
$$E_i\,\supseteq \,E_{i+1}$$ for all $i$ such that $E_i\,=\,0$ for $i\,\gg \,0$ and $E_i\,=\,E$ for
$i\,\ll \,0$. Indeed, given such a filtration,
we define an $\SO_{X\times \Spec(\mathbb{C}[t])}$-module as $\SE\,:=\,\bigoplus_{i\in \ZZ}E_it^{-i}$. 
Then the restriction of $\SE$ to the slice $X\times \{t\}$ is isomorphic to $E$ if $t\,\neq\, 0$ and
it is isomorphic to the associated graded object $\op{gr}{E_\bullet}$ if $t\,=\,0$
(see \cite[Lemma 1.10]{He} for more details).

A calculation shows the following (see \cite[\S~1.E.c]{He}):
$$
\operatorname{wt}(f^* \SL_{\det}^a) \ =\ 2 a \sum \big( \rk(E)\deg(E_l)-\rk(E_l)\deg(E) \big)
$$
and the proposition follows.
\end{proof}

Therefore, the substack $\SM^{\ssvb}(X,r,\xi)$ of semistable vector bundles can be intrinsically
recovered from $\SM(X,r,\xi)$. More precisely, we have:

\begin{corollary}\label{cor:intrinsic}
Let $X$ be a smooth projective curve of any genus. 
Let $\SL$ be a line bundles on $\SM\,=\,\SM(X,r,\xi)$ such that the substack of $\SL$-semistable points satisfies
the condition $\emptyset \,\subsetneq\, \SM^{\SL-\text{ss}} \,\subsetneq\, \SM$. Let $\SL'$ be another such line
bundle. Then $\SM^{\SL-\text{ss}}\,=\,\SM^{\SL'-\text{ss}}$, and this is the substack 
$\SM^{\ssvb}(X,r,\xi)$ of semistable vector bundles in the usual sense.
\end{corollary}

\begin{proof}
This follows immediately from Proposition \ref{L-semistability}.
\end{proof}

Alternatively, the substack of semistable vector bundles can also be recovered using a result of 
Faltings \cite[Theorem I.3]{Fa93} (see also \cite[Proposition 1.6.2]{Ray} 
and \cite[Theorem 6.2 and Lemma 
8.3 by Nori]{Se}) which identifies the complement of $\SM^{\ssvb}(X,r,\xi)$ in $\SM(X,r,\xi)$ as the 
substack of $k$-points on which all sections of powers of the generator of the determinant of the 
cohomology line bundle vanish. See also recent works of Weissmann-Zhang for another approach \cite{WZ}.

Once we recover the substack parametrizing the semistable locus, we can apply 
\cite[Theorem 3.12]{ABBLT} to construct a {\em good moduli space} (in the sense of J. Alper \cite{Alp}) 
${M}^{\ssvb}(X,r,\xi)$ of $\SM^{\ssvb}(X,r,\xi)$ and a map $$\SM^{\ssvb}(X,r,\xi) \,
\,\longrightarrow \,\,{M}^{\ssvb}(X,r,\xi),$$
which coincides with the usual moduli space of semistable vector bundles.

\begin{proposition}\label{prop:beyondGIT}
Let $X$ and $X'$ be smooth complex projective curves of any genus
and $r,\, r'\,>\,1$. If $\SM(X,r,\xi)\,
\cong\, \SM(X',r',\xi')$, then $M^{\ssvb}(X,r,\xi)\,\cong\, M^{\ssvb}(X',r',\xi')$.
\end{proposition}

\begin{proof}
Assume that we have an isomorphism
$$\Psi\,:\,\SM(X,r,\xi) \,\,\longrightarrow\,\, \SM(X',r',\xi').$$
Let $\SL'\,=\,\SL'_{\det}$ be the determinant line bundle on $\SM(X',r',\xi')$,
and let $\SL\,=\,\Psi^*\SL'$. Using the definition of $\SL$-semistability,
it is easy to check that $\Psi$ restricts to an isomorphism between
$\SM^{\SL-\text{ss}}(X,r,\xi)$ and 
$\SM^{\SL'-\text{ss}}(X',r',\xi')\,=\, \SM^{\ssvb}(X',r',\xi')$.
By Corollary \ref{cor:intrinsic} we
obtain that $\SM^{\SL-\text{ss}}(X,r,\xi)\,=\,\SM^{\ssvb}(X,r,\xi)$.
Therefore, $\Phi$ restricts to an isomorphism
$$\Psi^{\ssvb}\,:\,\SM^{\ssvb}(X,r,\xi) \,\longrightarrow\, \SM^{\ssvb}(X',r',\xi')\, .$$
Let $\pi$ and $\pi'$ be the projections from each of these moduli stacks of semistable bundles to
the respective moduli schemes. Consider the composition of maps
$$\pi'\circ \Psi^{\ssvb}\, : \, \SM^{\ssvb}(X,r,\xi) \,\longrightarrow\, M^{\ssvb}(X',r',\xi')\, .$$
By \cite[Theorem 6.6]{Alp}, the good quotient $M^{\ssvb}(X,r,\xi)$ corepresents the moduli stack $\SM^{\ssvb}(X,r,\xi)$. Thus, the map $\pi'\circ \Psi^{\ssvb}$ factors through the moduli scheme $M^{\ssvb}(X,r,\xi)$:
\[
\xymatrix{
\SM^{\ssvb}(X,r,\xi) \ar[rr]^{\Psi^{\ssvb}} \ar[d]_{\pi} && \SM^{\ssvb}(X',r,\xi') \ar[d]^{\pi'}\\
M^{\ssvb}(X,r,\xi) \ar[rr]^{\psi} && M^{\ssvb}(X',r',\xi')
}
\]
As the inverse of $\Psi^{\ssvb}$ also descends, the above map $\psi$ is an isomorphism.
\end{proof}

From Proposition \ref{prop:beyondGIT} we can obtain the Torelli theorem for the moduli stacks applying any of the 
existing Torelli theorems for the moduli schemes. For instance, the following theorem results by
applying \cite[Corollary 2.12]{AB}.

\begin{theorem}
\label{thm:beyondGIT}
Let $X$ and $X'$ be smooth complex projective curves of genus at least 4. Suppose that $r,\,r'\,\ge\,
2$. Let $\xi$ and $\xi'$ be line bundles on $X$ and $X'$ respectively. Then
$\SM(X,r,\xi)\,\cong \,\SM(X',r',\xi')$ if and only if $X\,\cong\, X'$, $r\,=\,r'$ and
$\deg(\xi)\,\cong\, \pm \deg(\xi') \pmod{r}$.
\end{theorem}

\begin{proof}
If $\SM(X,r,\xi)\,\cong\, \SM(X',r',\xi')$, then Proposition \ref{prop:beyondGIT} implies that
$$M^{\ssvb}(X,r,\xi)\ \cong \ M^{\ssvb}(X',r',\xi')$$ 
and the theorem follows from \cite[Corollary 2.12]{AB}.
\end{proof}

Observe that the same argument can also be applied to prove a Torelli theorem for the moduli stack of 
rank $r$ bundles of fixed degree, invoking the appropriate Torelli theorem for moduli spaces. 

Similarly, we can consider principal $G$-bundles for any complex reductive group $G$. Let $\SM_G^d(X)$ be the 
connected component of moduli stack of principal $G$-bundles on $X$ corresponding to $d\,\in\, \pi_1(G)$ (the 
connected components of the moduli stack are parametrized by $\pi_1(G)$).

Recall that a principal $G$-bundle $E_G$ is semistable in the sense of Ramanathan if
for any reduction $P^Q\, \subset\, E_G$ to a parabolic subgroup $Q\subset G$ and
for any dominant character $\chi$ of $Q$, the degree of the associated
line bundle $P^Q(\chi)$ satisfies the inequality $\deg(P^Q(\chi))\,\leq\, 0$.

\begin{lemma}\label{recover-ss-gbundles}
Let $X$ be a smooth complex projective curve of any genus. Take a (rational) line bundle
$\SL\in\op{Pic}(\SM^d_G(X))\otimes \QQ$,
and let $$\SU_\SL \,=\, \SM_G^d(X)^{\SL-\textup{ss}}\,\subset\, \SM_G^d(X)$$ be the substack
of $\SL$-semistable principal $G$-bundles on $X$. Let $\SU$ be the intersections
of all $\SU_{\SL}$ which are nonempty. Then $\SU$ is the substack of
semistable principal $G$-bundles in the sense of Ramanathan. 

Furthermore, there exist a (rational) line bundle $\SL$ such that $\SU\,=\,\SU_{\SL}$.
\end{lemma}

\begin{proof}
Since the curve $X$ is fixed during this proof, we will drop it entirely
from the notation, denoting the moduli stack by just $\SM_G^d$. 
Let $Z'$ be the center of $[G,\,G]$. It is a finite group. 
A principal $G$-bundle is semistable in the sense of Ramanathan if and
only if its extension of structure group to a principal $G/Z'$-bundle
is semistable. We are going to see that the same holds for $\SL$-semistability
in the sense of Definition \ref{def:l-ss}.

Consider the morphism
$$
p\,:\,\SM^d_G \,\longrightarrow\, \SM^{d'}_{G/Z'}
$$
which sends a principal $G$-bundle on $X$ to the associated $G/Z'$-bundle.
Let $P$ be a principal $G$-bundle on $X$ mapping to a principal $G/Z'$-bundle
$P'$.

We claim that a morphism $f'\,:\,\Theta\,\too\, \SM^{d'}_{G/Z'}$ (with $f'(1)\,=\,P'$) can be lifted
to $$f\ :\ \Theta\ \too\ \SM^d_G$$ (with $f(1)\,=\,P$) 
after passing to a ramified cover $\Theta \, \stackrel{[n]}{\too}\,\Theta$
given by $t\,\longmapsto\, t^n$. Indeed, in \cite[1.F.b]{He} it is proved that
there is a bijection between morphisms $f\,:\,\Theta\,\too \, \SM_G$ and 
equivalence classes of pairs $(\lambda:\GG_m\to G,\, P_\lambda\subset P)$ consisting of a one parameter
subgroup $\lambda\,:\,\GG_m\,\too\, G$ and
a reduction of structure group $P_\lambda\,\subset\, P$ of a principal $G$-bundle $P$ to the parabolic subgroup
$$\{g\,\in\, G\,\,\big\vert\,\, \lim_{t\to 0}\lambda(t)g\lambda(t^{-1}) \,\text{\,exists}\}\,\subset\, G .$$
Two pairs are equivalent if $\lambda$ is conjugate by
an element of the parabolic subgroup. In this bijection, if 
the morphism $f(t)$ is replaced by $f(t^n)$, then the
the one-parameter subgroup $\lambda(t)$ is replaced by $\lambda(t^n)$,
and the parabolic subgroup and reduction stay the same. 
Therefore a morphism $f\,:\,\Theta\,\too \,\SM^{d'}_{G/Z'}$ produces a 
one-parameter subgroup $\lambda'\,:\,\GG_m\,\too \,G/Z'$ and a reduction of structure
group of the principal $G/Z'$-bundle $P'$ to
the parabolic subgroup associated to $\lambda'$. The parabolic subgroups
of $G$ are the same as the parabolic subgroups of $G/Z'$, and 
a reduction of structure group of a principal $G$-bundle $P$
to a parabolic subgroup of $G$ induces a reduction, to 
the corresponding parabolic subgroup of $G/Z'$, of the principal $G/Z'$-bundle corresponding to $P$. 
On the other hand, since $Z'$ is a finite
abelian group, we have a Cartesian diagram 
$$
\xymatrix{
{\GG_m}\ar[r]^{\lambda}\ar[d]_{q} & {G}\ar[d]\\ 
{\GG_m}\ar[r]^{\lambda'} & {G/Z'}\\ 
}
$$ 
where $q$ is just the cover $t\,\longmapsto\, t^n$ for some $n$. Therefore,
$\lambda'$ can be lifted to $G$ after passing to a cover of order $n$.
This implies that $f$ can be lifted to $\SM^d_G$ as claimed.

By \cite[Def 5.2.1 and Thm 5.3.1]{BH10}
the morphism $p$ induces an isomorphism 
$
p^*\,:\,
\op{Pic}(\SM^{d'}_{G/Z'} )\otimes \QQ
\,\longrightarrow\, 
\op{Pic}(\SM^d_G )\otimes \QQ 
$.

Therefore, a point 
in $\SM^d_G $ is $\SL$-semistable 
if and only if its image 
in $\SM^{d'}_{G/Z'} $ is $\SL'$-semistable, where $p^*(\SL')\,\cong\, \SL$. 

Let $Z$ be the center of $G$. 
Note that $G/Z'\,=\,G/[G,\,G] \times G/Z$. Therefore, we have
\begin{equation}\label{eq:prod}
\SM^{d'}_{G/Z'} \,\,=\,\,\SM^{d_1}_{G/[G,\,G]}\times \SM^{d_2}_{G/Z} \, .
\end{equation}
The group $G/[G,\,G]$ is a torus
(isomorphic to $\mathbb{G}_m^r$), and the group $G/Z$ is semisimple and of adjoint type.
The global functions on $\SM^{d'}_{G/G'}$ are just the constant scalars $\CC$,
and $\op{Pic}(\SM^{d_2}_{G/Z})$ is discrete (by \cite[Theorem 5.3.1]{BH10}), so
\cite[Lemma 2.1.4]{BH10} gives:
$$
\op{Pic}(\SM^{d'}_{G/Z'}) \,=\,\op{Pic}(\SM^{d_1}_{G/[G,\,G]})\oplus \op{Pic}(\SM^{d_2}_{G/Z}) \, .
$$
Therefore, a line bundle $\SL$ on $\SM^{d'}_{G/Z'}$ is of the form
$\SL_1\boxtimes \SL_2$, and a point $x$ in $\SM^{d'}_{G/Z'}$
is $\SL$-semistable if and only if both the projections 
$x_1\,\in\, \SM^{d_1}_{G/[G,\,G]}$ and $x_2\,\in\, \SM^{d_2}_{G/Z}$ are, 
respectively, $\SL_1$-semistable and $\SL_2$-semistable. In other words,
$$
\SM^{d'}_G{}^{\SL-\text{ss}}\,=\, \SM^{d_1}_{G/[G,\,G]}{}^{\SL_1-\text{ss}}\times 
\SM^{d_2}_{G/Z}{}^{\SL_2-\text{ss}} \, .
$$

The torus $G/[G,\,G]$ is the product $\GG_m^s$. Then
$$
\SM^{d_1}_{G/[G,\,G]} \,\cong\, B(\GG_m{}^{\times s}) \times J^{\times s}
$$
where $J$ is the Jacobian scheme of the curve. The scheme
$J^{\times r}$ is projective, so the global functions are just 
the scalars $\CC$,
and $\op{Pic}(B(\GG_m{}^{\times r})\,=\,\ZZ^r$ is discrete, so
applying \cite[Lemma 2.1.4]{BH10} again we get that
$$
\op{Pic}(\SM^{d_1}_{G/[G,\,G]}) \,\cong\, \ZZ^{r} \oplus \op{Pic}(J^{\times r})
$$
and then a line bundle on $\SM^{d_1}_{G/[G,\,G]}$ is of the form 
$\SL_1\,=\,\SL_{1,1}\boxtimes \SL_{1,2}$, where 
$\SL_{1,1}\,\in\, \op{Pic}B(\GG_m{}^{\times r})\,=\, \ZZ^r$
and $\SL_{1,2}\,\in \, \op{Pic}(J^{\times r})$, and the point 
$x_1$ is $\SL_1$-semistable if and only if both $x_{1,1}$ and $x_{1,2}$
are respectively $\SL_{1,1}$-semistable and $\SL_{1,2}$-semistable.
The point $x_{1,2}\,\in\, J^{\times s}$ is automatically $\SL_{1,2}$-semistable,
because $J^{\times s}$ is a scheme, and hence any morphism from $\Theta$ into
it is trivial.
A line bundle on $B\GG_m$ is the same thing as a one dimensional vector
space with an action of $\GG_m$. Therefore, $\op{Pic}(B\GG_m)=\ZZ$. Let $L_a$ be a line bundle on $B\GG_m$. The morphisms from $\Theta$ to $B\GG_m$
are classified by $\ZZ$, because such a morphism is equivalent to an equivariant
line bundle on $\AA^1$, and these are classified by the weight of the action
on the fiber over zero. Let $f_b:\Theta \too B\GG_m$ be the morphism 
corresponding to $b\in \ZZ$. Then ${\rm wt}(f_b^*L_a|_0)=ab$. Therefore
(Definition \ref{def:l-ss}), the point in $B\GG_m$ is $L_a$-semistable if and only $a$ is zero.

It follows that any point $x_{1,1}\,\in\, B(\GG_m{}^{\times r})$ is $\SL_{1,2}$-semistable if and only if all
the coordinate of $\SL_{1,2} \,\in\, \ZZ^r$ are zero. Therefore, to prove this lemma
we may assume that all these coordinates are zero.

Hence, a point $x$ in the stack \eqref{eq:prod} is $\SL$-semistable if
and only if the projection to $\SM_{G/Z} $ is 
$\SL_2$-semistable. The same holds for semistability in the sense of Ramanathan.

Therefore, we may assume that $Z$ is trivial and $G$ is a product of
simple groups of adjoint type: $G\,=\,G_1\times \cdots\times G_s$.
Using \cite[Definition 5.2.1, Remark 4.3.3 and Theorem 5.3.1]{BH10},
$\op{Pic}(\SM_{G_1\times \cdots\times G_s} )\otimes \QQ \,\cong\, \QQ^s$, 
where $s$ is the number
of simple factors, and the generators of this group come from pullbacks
of line bundles on each factor $\SM_{G_i} $. In other words,
a line bundle $\SL$ in $\op{Pic}(\SM_{G/Z} )\otimes \QQ$ is of
the form $\SL_1\boxtimes \cdots \boxtimes \SL_s$, where $\SL_i$ is
a (rational) line bundle on $\SM_{G_i}$. It is easy to check
that a point $x$ in $\SM_{G_1\times \cdots\times G_s}$ is $\SL$-semistable if and only if
all projections $x_i\,\in\, \SM_{G_i}$ are $\SL_i$-semistable:
$$
\SM^{\SL-\textup{ss}}_{G_1\times \cdots\times G_s}\, = \,
\SM^{\SL_1-\textup{ss}}_{G_1}\times \cdots\times \SM^{\SL_s-\textup{ss}}_{G_s}
$$
and, again, the same holds for semistability in the sense of Ramanathan.

By \cite[Theorem 5.3.1]{BH10}, $\op{Pic}(\SM_{G_i})\otimes \QQ\,\cong\, \QQ$
and the determinant line bundle $\SL_{\det}$, whose fiber over $P$ is
$(\det H^1 (\text{ad}(P))) \otimes (\det H^0(\text{ad}(P)))^{-1}$,
is a generator. As in the case of
vector bundles, if follows from 
Remark \ref{different-l}, that it is enough to consider three cases:
If $\SL_i$ is a positive multiple
of the determinant bundle, then $\SL_i$-semistability is equivalent
to the usual notion of semistability defined by Ramanathan 
(see \cite[\S~1.F]{He}).
If $\SL_i$ is a negative multiple of the determinant, then the substack
of $\SL_i$-semistable points is empty, and if $\SL_i$ is trivial, then
the substack of $\SL_i$-semistable points is the whole moduli stack.

Therefore, $\SM_{G}^{\SL-\textup{ss}}$
is smallest and non-empty when $\SL_i$ is a positive multiple of the
determinant bundle on $\SM_{G_i}$ for all $i$. Furthermore, 
for such $\SL_i$, a point $x\,\in\, \SM_G$ is 
$\SL\,=\,\SL_1\boxtimes \cdots \boxtimes \SL_s$-semistable 
if and only if it is semistable in the usual sense of Ramanathan.
\end{proof}

\begin{corollary}
\label{cor:TorellliG-bundles}
Let $X$ and $X'$ be smooth projective complex curves with genera $g(X),\, g(X')\,\geq\, 3$ respectively.
Let $G$ and $G'$ be algebraic connected reductive complex groups.
If the moduli stacks $\SM^d_G(X)$ and $\SM^{d'}_{G'}(X')$ are
isomorphic as stacks, then the curves $X$ and $X'$ are also isomorphic.
\end{corollary}

\begin{proof}
Arguing as done for Proposition \ref{prop:beyondGIT} we obtain that the 
corresponding moduli schemes of principal bundles are isomorphic,
and then we apply \cite[Theorem 0.1]{BH12}.
\end{proof}

\section{Torelli for moduli stack of rank 2 vector bundles}\label{section:rank2}

In this section we prove a Torelli theorem for the moduli stack of rank 2 vector bundles with trivial 
determinant. Notice that for genus 2 curves the moduli space of rank 2 vector bundles with trivial 
determinant is isomorphic to $\PP^3$, irrespective of the curve. 
Nevertheless, we will show that the geometry of the moduli stack, contrary to the scheme, does 
indeed encode the geometry of the curve effectively.

Throughout this section, $M^{\svb}(X,2,\SO_X)\,\subset\, M^{\ssvb}(X,2,\SO_X)$ denotes the subset of
stable bundles and $S(X,2,\SO_X)\,:=\,M^{\ssvb}(X,2,\SO_X) \backslash M^{\svb}(X,2,\SO_X)$
is the subset of strictly semistable vector bundles.

\begin{lemma}
The image of the set of non-simple points in $\SM^{\ssvb}(X,2,\SO_X)$ under the quotient
map $\SM^{\ssvb}(X,2,\SO_X)\, \longrightarrow\, M^{\ssvb}(X,2,\SO_X)$ coincides with the set of
strictly semistable vector bundles $S(X,2,\SO_X)\,:=\,M^{\ssvb}(X,2,\SO_X) \backslash M^{\svb}(X,2,\SO_X)$.
\end{lemma}

\begin{proof}
Since the preimage of $M^{\svb}(X,2,\SO_X)$ under the quotient map is the substack of stable vector
bundles, and all stable vector bundles are simple, the image of each non-simple vector bundle in 
$\SM^{\ssvb}(X,2,\SO_X)$ must be a strictly semistable vector bundle. Let us prove that the non-simple
vector bundles surject onto the strictly semistable ones.

For each strictly semistable vector bundle $E$ in $S(X,2,\SO_X)$ there exists a polystable
vector bundle $\widetilde{E}\,=\,L\oplus L^{-1}$ which is $S$-equivalent to $E$. The map
$$\begin{pmatrix}
\lambda & 0\\
0 & -\lambda
\end{pmatrix}\, : \, L\oplus L^{-1} \, \longrightarrow L\oplus L^{-1}$$
is a nontrivial traceless endomorphism of $\widetilde{E}$, so $\widetilde{E}$ represents
a non-simple point in $\SM^{\ssvb}(X,2,\SO_X)$ which projects to $E$.
\end{proof}

\begin{theorem}
\label{thm:rank2}
Let $X$ and $X'$ be two irreducible smooth complex projective curves of genus $g$ and $g'$
respectively with $g,\,g'\,\ge\, 2$. If $\SM(X,2,\SO_X)\,\cong\, \SM(X',2,\SO_{X'})$, then $X\,\cong\, X'$.
\end{theorem}

\begin{proof}
Repeating the argument in the previous section we know that the isomorphism $\SM(X,2,\SO_X)\,\cong\, 
\SM(X',2,\SO_{X'})$ restricts to an isomorphism of the semistable locus $\SM^{\ssvb}(X,2,\SO_X)\,\cong\, 
\SM^{\ssvb}(X',2,\SO_{X'})$ and that this map descends to an isomorphism
\[
\xymatrix{
\SM^{\ssvb}(X,2,\SO_X) \ar[rr] \ar[d]_{\pi} && \SM^{\ssvb}(X',2,\SO_{X'}) \ar[d]^{\pi'}\\
M^{\ssvb}(X,2,\SO_X) \ar[rr]^{\psi} && M^{\ssvb}(X',2,\SO_{X'})
}
\]
such that $\psi(S(X,2,\SO_X))\,=\,S(X',2,\SO_{X'})$. Let $K(X)\,=\,J(X)/\{\pm 1\}$ denote the quotient of
$J(X)$ by the inversion map $i$ defined by $L\,\longmapsto\, L^{-1}$. In particular, if the genus of $X$
is two, then $K(X)$ is the Kummer surface associated to the Jacobian. Each
S-equivalence class of a bundle $E$ in $S(X,2,\SO_X)$ has a unique representative of the form
$E\,=\,L\oplus L^{-1}$, so there exists a correspondence between the points of $S(X,2,\SO_X)$ and
the points of $K(X)$. By \cite[Theorem 2]{NR69}, the moduli space $M^{\ssvb}(X,2,\SO_X)$ is
isomorphic to $\PP(H^0(J(X),\, \SL_\theta^2))$, where $\SL_\theta$ is the canonical polarization of
the Jacobian induced by the natural embedding of $X$ in $J(X)$. Now Proposition 6.3 and the construction from
Theorem 2 of \cite{NR69} prove that the map from $K(X)$ to $\PP(H^0(J(X),\, \SL_\theta^2))
\,\cong\, M^{\ssvb}(X,2,\SO_X)$, which sends the class of $L$ to the S-equivalence class of $L\oplus L^{-1}$,
gives an embedding $K(X)\, \hookrightarrow\, M^{\ssvb}(X,2,\SO_X)$ whose image is the subvariety $S(X,2,\SO_X)$ preserved by $\psi$.

From the geometry of $K(X)$ we can reconstruct the map $J(X)\longrightarrow K(X)$ canonically as follows. 
First, remove the singular points of $K(X)$. Let $K^{\op{sm}}(X)$ be the smooth part. The fundamental 
group $\pi_1(K^{\op{sm}}(X))$ has a unique maximal torsion free subgroup. This subgroup coincides with
the subgroup $\pi_1(J(X)\backslash J(X)[2])$, where $J(X)[2]$ denotes the 2-torsion part of the Jacobian,
and the quotient group is $\ZZ/2\ZZ$. The corresponding double covering is then $J(X)\backslash J(X)[2] \longrightarrow 
K^{\op{sm}}(X)$. Now $J(X)$ is the unique abelian compactification of $J(X)\backslash J(X)[2]$, and the 
map $J(X)\,\longrightarrow\, K(X)$ is the unique possible extension to $J(X)$ for the double cover 
$J(X)\backslash J(X)[2] \,\longrightarrow \,K^{\op{sm}}(X)$. Thus, the isomorphism $\psi$ induces an 
isomorphism $J(X)\,\cong\, J(X')$.

Now, consider the compositions of maps
$$j_X\,:\,J(X)\,\longrightarrow\, K(X)\,\hookrightarrow\, M^{\ssvb}(X,2,\SO_X)\, ,$$ 
$$j_{X'}\,:\,J(X') \,\longrightarrow\, K(X') \,\hookrightarrow\, M^{\ssvb}(X',2,\SO_{X'})\, .$$ Let $\SL$ be the 
ample generator of $M^{\ssvb}(X,2,\SO_X)$. By construction \cite{NR69}, it follows that $j^*\SL$ is a multiple of the 
canonical polarization of $J(X)$. Since $\op{Pic}(M^{\ssvb}(X,2,\SO_{X})\,\cong\, 
\op{Pic}(M^{\ssvb}(X',2,\SO_{X'})\,=\,\ZZ$ by \cite{DN}, $\SL'\,:=\,(\psi^{-1})^*\SL$ is a multiple of the ample 
generator of $M^{\ssvb}(X',2,\SO_{X'})$, so $j_{X'}^*(\SL')$ is also a multiple of the canonical 
polarization. Thus, the isomorphism $J(X)\,\cong\, J(X')$ induced by $\psi$ is an isomorphism of canonically 
polarized Jacobians. By the classical Torelli Theorem, $X\,\cong\, X'$.
\end{proof}

\begin{remark}\label{rmk:rank2}
The above proof also shows that the Torelli theorem holds for the substack of semistable vector
bundles. If $\SM^{\ssvb}(X,2,\SO_X)\,\cong\, \SM^{\ssvb}(X',2,\SO_{X'})$, then $X\,\cong\, X'$. Thus,
the isomorphism class of the curve $X$ cannot be recovered from the moduli scheme $M^{\ssvb}(X,2,\SO_X)$, when
$g\,=\,2$, but it can be recovered from the geometry of the moduli stack $\SM^{\ssvb}(X,2,\SO_X)$.
\end{remark}

\section{Proof of the Torelli theorem}
\label{section:mainThm}

In this section we will combine the previous cases to obtain a 
Torelli theorem for the moduli stack of vector bundles for curves of any rank $r\,\ge\, 2$ and any genus 
$g\,\ge\, 2$. Let us start with a basic dimensional computation, which will allow us to apply the appropriate Torelli 
theorems selectively.

\begin{lemma}\label{lemma:dimensions}
Let $M$ be some variety which is isomorphic to a moduli space of semistable vector bundles of
rank $r\,\ge\, 2$ and determinant $\xi$ over a smooth complex projective curve $X$ of genus $g\,\ge\, 2$. Then, either
\begin{enumerate}
\item $\dim(M)\,=\,3$, in which case $g\,=\,2$ and $r\,=\,2$,
\item $\dim(M)\,=\,6$, in which case $g\,=\,3$ and $r\,=\,3$,
\item $\dim(M)\,=\,8$, in which case $g\,=\,2$ and $r\,=\,3$,
\item $\dim(M)\,\ge\, 9$, in which case either
\begin{itemize}
\item $g\,\ge\, 4$, or 
\item $g\,=\,3$ and $r\,\ge\, 3$, or
\item $g\,=\,2$ and $r\,\ge\, 4$.
\end{itemize}
\end{enumerate}
\end{lemma}

\begin{proof}
The dimension of a moduli space of a curve of genus $g\,\ge\, 2 $ and rank $r
\,\ge \,2$ is $d_{g,r}\,=\,(r^2-1)(g-1)$. Clearly the dimensions
$$d_{2,2}\,=\,3, \quad \quad d_{2,3}\,=\,8, \quad \text{and} \quad d_{3,2}\,=\,6$$
are distinct and less than 9. Thus, if we prove that the dimension of any other moduli space with
$(g,r)\,\not\in\, \{(2,\,2),\,(2,\,3),\,(3,\,2)\}$ is greater or equal to 9, there will exist
a unique option for the previous given dimensions and the lemma will follow. Clearly $d_{g,r}$
is increasing in $r$ and $g$. Thus, if $g\,\ge\, 4$ then as $r\,\ge\, 2$, we have
$$d_{g,r}\,=\,(r^2-1)(g-1)\,\ge\, (2^2-1)(3-1)\,=\, 9 $$
and, if $r\,\ge\, 4$, then as, $g\,\ge\, 2$, we have
$$d_{g,r}\,=\,(r^2-1)(g-1)\,\ge\, (4^2-1)(2-1)\,=\,15\,>\,9.$$
Finally, $d_{3,3}\,=\,16\,>\,9$ and we obtain the following table for the values of $d_{g,r}$ which
proves the result:
\begin{center}
\begin{tabular}{c|ccc}
$g\backslash r$ & 2 & 3 & $\ge 4$\\
\hline
2 & 3 & 8 & $\ge 15$\\
3 & 6 & 16 & $\ge 30$\\
$\ge 4$ & $\ge 9$ & $\ge 24$ & $\ge 45$\\
\end{tabular}
\end{center}
This completes the proof.
\end{proof}

\begin{theorem}
\label{thm:main}
Let $X$ and $X'$ be two irreducible smooth complex projective curves of genus $g$ and $g'$
respectively, with $g,\,g'\,\ge\, 2$. Let $r,\,r'\,\ge\, 2$,
and fix line bundles $\xi$ and $\xi'$ on $X$ and $X'$ respectively. Let
$$\Psi\,:\,\SM(X,r,\xi) \,\longrightarrow\, \SM(X',r',\xi')$$ be an
isomorphism between the corresponding moduli stacks of vector bundles with fixed determinant. Then
$r\,=\,r'$ and $X\,\cong\, X'$.
\end{theorem}

\begin{proof}
Repeating the argument from Section \ref{section:beyondGIT} and applying Proposition \ref{prop:beyondGIT}, the
above map $\Psi$ induces an isomorphism $M^{\ssvb}(X,r,\xi) \,\cong\, M^{\ssvb}(X',r',\xi')$
between the moduli spaces. By Lemma 
\ref{lemma:dimensions}, the dimension of this moduli space is either $3$, $6$, $8$ or at least $9$. Let 
us consider each case individually.

\emph{\underline{Dimension 3:}}

By Lemma \ref{lemma:dimensions}, $g\,=\,g'\,=\,2$ and $r\,=\,r'\,=\,2$. By \cite{NR69}, there are two possible
different geometries for these moduli spaces. Either the moduli spaces are both isomorphic to $\PP^3$, in which
case $\deg(\xi)$ and $\deg(\xi')$ are even, or both the moduli spaces are isomorphic to an intersection of
quadrics in $\PP^5$, in which case $\deg(\xi)$ and $\deg(\xi')$ are odd.

If $\deg(\xi)$ and $\deg(\xi')$ are both odd then we can apply the Torelli Theorems for rank 2 bundles with fixed 
determinant with odd degree by Mumford and Newstead \cite[Corollary p.1201]{MN68} or by Tyurin \cite[Theorem 
1]{Tyu69}.

Otherwise, if $\deg(\xi)$ and $\deg(\xi')$ are both even, then there exist line bundles $L$ and $L'$ on $X$ and 
$X'$ respectively such that $L^2\,=\,\xi$ and $(L')^2\,=\,\xi'$. In that case, the maps $E\,\longmapsto\,
E\otimes L^{-1}$ and 
$E'\,\longmapsto\, E'\otimes L^{-1}$ induce isomorphisms of stacks $\SM(X,2,\xi) \,\cong\, \SM(X,2,\SO_X)$ and 
$\SM(X',2,\xi')\,\cong\, \SM(X',2,\SO_{X'})$ respectively. Thus, we have an isomorphism of stacks $\SM(X,2,\SO_X) 
\,\cong\, \SM(X',2,\SO_{X'})$ and now we can apply Theorem \ref{thm:rank2} to conclude that $X\,\cong\, X'$.

\emph{\underline{Dimension 6:}}

By Lemma \ref{lemma:dimensions}, $g\,=\, 3\,=\, g'$ and $r\,=\, 2\,=\, r'$. So we can apply the Torelli 
Theorem by Kouvidakis and Pantev for curves of genus $g\ge 3$ of the same rank \cite[Theorem E]{KP95} 
to conclude that $X\,\cong\, X'$.

\emph{\underline{Dimension 8:}}

By Lemma \ref{lemma:dimensions} we have $g\,=\, 2\,=\, g'$ and $r\,=\, 3\,=\, r'$. By \cite[Theorem 1.8]{BBGN07}, the 
cohomological Brauer group of the moduli space $M^{\ssvb}(X,3,\xi)$ is 
$$\op{Br}\left(M^{\ssvb}(X,3,\xi)\right)\,:=\, H^2\left(M^{\ssvb}(X,3,\xi)_{et},\mathbb{G}_m\right)\,\cong\, 
\mathbb{Z}/(\op{g.c.d}(r,\deg(\xi)))\mathbb{Z}$$ so it is either $0$ (when $\deg(\xi)$ is coprime to $3$) or it is 
$\ZZ/3\ZZ$ (when $\deg(\xi)$ is a multiple of 3). As $\op{Br}(M^{\ssvb}(X,3,\xi)\,\cong\, 
\op{Br}(M^{\ssvb}(X',3,\xi')$, either $\deg(\xi)$ and $\deg(\xi')$ are both coprime to 3 or they are both 
multiples of 3.

If $\deg(\xi)$ and $\deg(\xi')$ are both coprime to 3, then we can apply the Torelli theorems by Tyurin 
\cite[Theorem 1]{Tyu70} or Narasimhan-Ramanan \cite[Theorem 3]{NR75}.

If $\deg(\xi)$ and $\deg(\xi')$ are both multiples of 3, then, as before, there exist line bundles $L$ and
$L'$ over $X$ and $X'$ respectively such that $L^3\,=\,\xi$ and $(L')^3\,=\,\xi'$. The maps $E\,\longmapsto 
\,E\otimes L^{-1}$ and $E'\,\longmapsto\, E'\otimes L^{-1}$ induce isomorphisms of moduli schemes 
$M^{\ssvb}(X,3,\xi) \,\cong\, M^{\ssvb}(X,3,\SO_X)$ and $M^{\ssvb}(X',3,\xi')\,\cong\,
M^{\ssvb}(X',3,\SO_{X'})$ respectively. 
Thus, we have an isomorphism of moduli schemes $M^{\ssvb}(X,3,\SO_X)\,\cong\, M^{\ssvb}(X',3,\SO_{X'})$.
Now we can apply the Torelli theorem for genus 2 curves and rank 3 bundles with trivial determinant
by Nguyen \cite[Corollary 3.4.4]{Ngu07} to obtain that $X\,\cong\, X'$.

\emph{\underline{Dimension at least 9:}}

By Lemma \ref{lemma:dimensions}, both $(g,\,r)$ and $(g',\,r')$ satisfy the conditions on the ranks and 
genera in Theorem \ref{thm:HitchinMap}, and hence we obtain that $X\,\cong\, X'$.
\end{proof}

\begin{remark}\label{rmk:modulischeme}
Using Lemma \ref{lemma:P1Discriminant} --- instead of Lemma \ref{lemma:recoverDiscriminant2} --- in the proof of 
Corollary \ref{cor:recoverDiscriminant} we can show that the discriminant $\SD$ is the closure of the 
image under $\SH$ of the set of rational curves in $T^*M^{\ssvb}(X,r,\xi)$ for any curve $X$ of genus 
$g\,\ge \,2$ and any rank $r\,\ge \,2$ such that $(g,\,r)\,\not\in\, \{(2,\,2),\,(2,\,3),\,(3,\,2)\}$.

Using this intrinsic characterization of $\SD$, the proof of the Torelli Theorem for the moduli scheme of 
vector bundles by Biswas, G\'omez and Mu\~noz \cite[Theorem 4.3]{BGM13} extends to curves of genus $g\,\ge\, 2$ 
and $r\,\ge\, 2$ where $(g,\,r)\,\not\in \,\{(2,\,2),\,(2,\,3),\,(3,\,2)\}$.

Then, the argument of Theorem \ref{thm:main} can also be used to prove the following Torelli
Theorem for the moduli scheme of vector bundles. Let $X$ and $X'$ be smooth complex projective curves
of genus at least 2, and let $r,\,r'\,\ge\, 2$. If $M^{\ssvb}(X,r,\xi)\,\cong\, M^{\ssvb}(X',r',\xi')$, then
$r\,=\,r'$ and either
\begin{itemize}
\item $X\,\cong\, X'$, or
\item $X$ and $X'$ are any pair of curves of genus 2 with $r\,=\,r'\,=\,2$ and $\deg(\xi)$ and $\deg(\xi')$
being even; in this case we have $M^{\ssvb}(X,2,\xi)\,\cong\, M^{\ssvb}(X',2,\xi')\,\cong\, \mathbb{P}^3$.
\end{itemize}

As a consequence, the unique case for genus and rank at least 2 where the moduli scheme 
of vector bundles does not admit a Torelli theorem is when the genus is 2, the rank is 2 and the degree 
of the determinant is even.
\end{remark}

\section*{Acknowledgments}

We would like to thank the anonymous referee for providing suggestions which helped in improving this paper.
This research was supported by grants PID2019-108936GB-C21, PID2022-142024NB-I00, RED2022-134463-T and CEX2019-000904-S funded by MCIN/AEI/ 10.13039/501100011033.
T. G. thanks Andr\'es Fern\'andez Herrero for helpful discussions. S.M. acknowledges support
of the DAE, Government of India, under Project Identification No. RTI4001. I.B. is partially supported by a J. C. Bose Fellowship (JBR/2023/000003).

\end{document}